%% file: aa_positiveness_full_main.tex
\documentclass[a4,11pt]{article}
\usepackage{graphicx}
\usepackage{amssymb}
\usepackage{url}
\usepackage{latexsym}
\usepackage{amssymb}
\usepackage{here}
\usepackage{wrapfig}
\usepackage{fancybox}
\usepackage{ascmac}
\usepackage{amsmath}
\usepackage{multicol}
\usepackage{mathtools}
\usepackage{amsthm}
\usepackage{cases}
\newtheorem{theo}{Theorem}[section]
\newtheorem{lem}[theo]{Lemma}
\newtheorem{rem}[theo]{Remark}
\newtheorem{coro}[theo]{Corollary}

\DeclareMathOperator*{\esssup}{ess\,sup}

\allowdisplaybreaks[4]

\begin{document}
\begin{center}
\noindent{\bf\Large Numerical verification method for positiveness of solutions to elliptic equations
}\vspace{10pt}\\
 {\normalsize Kazuaki Tanaka$^{1,*}$, Kouta Sekine$^{2}$, Shin'ichi Oishi$^{2,3}$}\vspace{5pt}\\
 {\it\normalsize $^{1}$Graduate School of Fundamental Science and Engineering, Waseda University, Japan\\
 $^{2}$Faculty of Science and Engineering, Waseda University, Japan\\
 $^{3}$CREST, JST, Japan}
\end{center}
{\bf Abstract}.
In this paper, we propose a numerical method for verifying the positiveness of solutions to semilinear elliptic equations.
We provide a sufficient condition for a solution to an elliptic equation to be positive in the domain of the equation, which can be checked numerically without requiring a complicated computation.
We present some numerical examples.
 
\renewcommand{\thefootnote}{\fnsymbol{footnote}}
\footnote[0]{{\it E-mail address:} $^{*}$\texttt{ imahazimari@fuji.waseda.jp}\\[-3pt]}
\renewcommand\thefootnote{*\arabic{footnote}}

\input{intro.tex}

\input{positiveness.tex}

\input{example.tex}

\appendix
\input{appendix_1.tex}

\input{embedding.tex}

\section*{Acknowledgements}
  The first author (K.T.) is supported by 
  the Waseda Research Institute for Science and Engineering, the Grant-in-Aid for Young Scientists (Early Bird Program).
  The second author (K.S.) is supported by JSPS KAKENHI Grant Number 16K17651. 

\bibliographystyle{plain} 
\bibliography{ref}

\end{document}

%% file: intro.tex
\section{Introduction}
We are concerned with verified numerical computation methods for solutions to the following elliptic problem:
\begin{subnumcases}
{\label{positiveproblem}}
-Lu(x)=f(u(x)), &$x\in \Omega$,\label{eq}\\
u(x)>0, & $x\in \Omega$\label{positivecond}
\end{subnumcases}
with an appropriate boundary condition, e.g., the Dirichlet type
\begin{align}
u(x)=0,~~x \in \partial\Omega,\label{diri}
\end{align}
where $\Omega$ is a bounded domain (i.e., an open connected bounded set) in $\mathbb{R}^{n}$ ($ n=1,2,3,\cdots$),
$f:\mathbb{R} \rightarrow \mathbb{R}$ 
is a given nonlinear function, and $L$ is a uniformly elliptic self-adjoint operator from its domain $D(L)$ to $L^{2}\left(\Omega\right)$ (the domain $D(L)$ depends on the smoothness of the boundary $\partial\Omega$).
More precisely, $L$ is written in the form
\begin{align}
L=\displaystyle \sum_{i,j=1}^{n}a_{i,j}\frac{\partial^{2}}{\partial x_{i}\partial x_{j}}+c,
\end{align}
where the following properties hold:
\begin{itemize}
\item$a_{i,j}\in L^{\infty}\left(\Omega\right) (i,j=1,2,\cdots,n)$ and $c\in L^{\infty}\left(\Omega\right)$;
\item$a_{i,j}=a_{j,i}$~$(i,j=1,2,\cdots,n)$;
\item There exists a positive number $\mu_{0}$ such that
\begin{align}
\displaystyle \sum_{i,j=1}^{n}a_{i,j}\left(x\right)\xi_{i}\xi_{j}\geq\mu_{0}\sum_{i=1}^{n}\xi_{i}^{2}
\end{align}
for all $ x\in\Omega$ and all $n$-tuples of real numbers $(\xi_{1},\xi_{2},\cdots,\xi_{n})$,
\end{itemize}
where $L^{\infty}\left(\Omega\right)$ is the functional space of Lebesgue measurable functions over $\Omega$ with the norm $\left\|u\right\|_{L^{\infty}\left(\Omega\right)}:=\esssup\{\left|u\left(x\right)\right|\,:\,x\in\Omega\}$ for $u\in L^{\infty}\left(\Omega\right)$.

Equation \eqref{eq}, including the case with \eqref{positivecond}, has been widely studied using analytical methods (see, e.g., \cite{lions1982existence,brezis1983positive,ouyang1999exact} and the references therein).
Moreover, verified numerical computation methods, which originate from \cite{nakao1988numerical, plum1991computer} and have been further developed by many researchers,
in recent years have turned to be effective to obtain, through computer-assistance, existence and multiplicity results for various concrete examples where purely analytical approaches have failed (see, e.g., \cite{nakao2001numerical, nakao2011numerical, plum2001computer, plum2008, takayasu2014remarks}).
These methods enable us to obtain a concrete ball containing exact solutions to \eqref{eq}, typically in the sense of the norms $\left\|\nabla\cdot\right\|_{L^{2}\left(\Omega\right)}$ and $\left\|\cdot\right\|_{L^{\infty}\left(\Omega\right)}$.
Therefore, such methods have the additional advantage that quantitative information of solutions to a target equation is provided accurately in a strict mathematical sense.
However, irrespective of how small the radius of the ball is, the positiveness of some solutions is not ensured without additional considerations.
For example, in the homogeneous Dirichlet case \eqref{diri}, it is possible for a solution that is verified by such methods not to be positive near the boundary $\partial\Omega$.

In this paper, we propose a numerical method for verifying the positiveness of solutions to \eqref{eq}, in order to verify solutions of (\ref{positiveproblem}).
Theorem \ref{theo1} provides a sufficient condition for positiveness,
which is a generalization of \cite[Theorem 2]{tanaka2015numerical} and \cite[Theorem 2.2]{tanaka2016sharp}.
Theorem \ref{theo1} (in this paper) enables us to numerically verify the positiveness for the whole of $\Omega$, and only requires a simple numerical computation.


%% file: positiveness.tex
\section{Verification of positiveness}
Throughout this paper, we omit the expression ``almost everywhere'' for Lebesgue measurable functions, for simplicity. For example, we employ the notation $u>0$ in the place of $u(x)>0$ a.e. $ x\in\Omega$.
Assuming that $H^{1}(\Omega)$ denotes the first order $L^{2}$-Sobolev space on $\Omega$, we define $H_{0}^{1}(\Omega):=\{u\in H^{1}(\Omega)\ :\ u=0~${\rm on}$~\partial\Omega${\rm ~in~the~trace~sense}$\}$.
Moreover, we assume that $f(u(\cdot)) \in L^2(\Omega)$
for each $u \in H^1(\Omega)$, and denote
\begin{align}
F:\left\{\begin{array}{lll}
V&\rightarrow&L^{2}\left(\Omega\right),\\
u&\mapsto&f(u(\cdot)).
\end{array}\right.\label{kataiF}
\end{align}

We introduce the following lemma that is required to prove Theorem \ref{theo1}.

\begin{lem}\label{main/lem}
Let $u\in H_{0}^{1}\left(\Omega\right)$ be a weak solution to \eqref{positiveproblem} with the boundary condition \eqref{diri},
such that
\begin{enumerate}
\item $F(u)\geq 0$ {\rm and} $F(u)\not\equiv 0 ;  $
\item $e(u(\cdot))<\infty$, where $e(x):=f(x)x^{-1},~x\in \mathbb{R}.$
\end{enumerate}
Then,
\begin{align}
\esssup\{e\left(u(x)\right)\,:\,x\in\Omega\}\geq\lambda_{1},\label{resoltlem}
\end{align}
where $\lambda_{1}$ is the first eigenvalue of the problem
\begin{align}
\left(-L\phi,v\right)_{L^{2}\left(\Omega\right)}=\lambda\left(\phi,v\right)_{L^{2}\left(\Omega\right)},~~\forall v\in H_{0}^{1}\left(\Omega\right);\label{weak/eig/pro}
\end{align}
the derivatives on the right side are understood in the sense of distributions.
\end{lem}

\begin{proof}
Let $\phi_{1}\geq 0\ (\phi_{1}\not\equiv 0)$ be the first eigenfunction corresponding to $\lambda_{1}$ (see, e.g., \cite[Theorems 1.2.5 and 1.3.2]{henrot2006extremum} for ensuring the nonnegativeness of the first eigenfunction).
Since $L$ is self-adjoint, it follows that
\begin{align*}
\left(F\left(u\right),\phi_{1}\right)_{L^{2}\left(\Omega\right)}=\lambda_{1}\left(u,\phi_{1}\right)_{L^{2}\left(\Omega\right)}.
\end{align*}
Therefore,
\begin{align*}
&\left(F\left(u\right),\phi_{1}\right)_{L^{2}\left(\Omega\right)}\\
=&\displaystyle \int_{\Omega}F\left(u\left(x\right)\right)u\left(x\right)^{-1}\left\{u\left(x\right)\phi_{1}\left(x\right)\right\}dx\\
\leq&\esssup\{e\left(u(x)\right)\,:\,x\in\Omega\}\left(u,\phi_{1}\right)_{L^{2}\left(\Omega\right)}\\
=&\lambda_{1}^{-1}\esssup\{e\left(u(x)\right)\,:\,x\in\Omega\}\left(F\left(u\right),\phi_{1}\right)_{L^{2}\left(\Omega\right)}.
\end{align*}
The positiveness of $\left(F\left(u\right),\phi_{1}\right)_{L^{2}\left(\Omega\right)}$ implies \eqref{resoltlem}.
\end{proof}

Using Lemma \ref{main/lem}, we are able to prove the following theorem, which provides a sufficient condition for the positiveness of solutions to \eqref{eq}.
\begin{theo}\label{theo1}
Let $u\in C^{2}\left(\Omega\right)\cap C(\overline{\Omega})$ be a function satisfying \eqref{eq} $($a boundary value condition is not necessary$)$ such that
$\underline{u}\leq u\leq\overline{u}$ in $\overline{\Omega}$ for $\underline{u},\ \overline{u}\in C(\overline{\Omega})$.
Function $u$ is always positive in $\Omega$, if the following conditions are satisfied:
\begin{enumerate}
\item$\underline{u}$ is positive in a nonempty subdomain $\Omega_{+}\subset\Omega\ (\Omega_{+}\neq\Omega) ;$
\item$\partial\Omega_{-}\cap\partial\Omega=\emptyset$, or $u=0$ {\rm on} $\partial\Omega_{-}\cap\partial\Omega$~if~$\partial\Omega_{-}\cap\partial\Omega\neq\emptyset$, where $\Omega_{-}$ is the interior of $\Omega\backslash\Omega_{+} ;$\label{theo1:boundary}
\item there exists a domain $\Omega_{-}^{\circ}\supset\Omega_{-}~ ( \Omega_{-}^{\circ}\neq\Omega_{-} )  $, s.t., $f\left([0,{\,}\max\{\overline{u}(x)_{\,}:_{\,}x\in\Omega_{-}^{\circ}\}]\right)\geq 0 ;$\label{theo1:f}
\item$f\left([\min\{\underline{u}(x)_{\,}:_{\,}x\in\Omega_{-}\},_{\,}0]\right)\leq 0$ and $f\left([\min\{\underline{u}(x)_{\,}:_{\,}x\in\Omega_{-}\},_{\,}0)\right)<0 ;$\label{theo1:f2}
\item For a domain $\hat{\Omega}$ such that $\Omega_{-}\subset\hat{\Omega}\subset\Omega,\ e([0,-\displaystyle \min\{\underline{u}(x)_{\,}:_{\,}x\in\Omega_{-}\}])<\lambda_{1}(\hat{\Omega})$, where $e(x):=f\left(x\right)x^{-1}$, and $\lambda_{1}(\hat{\Omega})$ is the first eigenvalue of the problem \eqref{weak/eig/pro} with the notational replacement $\Omega=\hat{\Omega}$.\label{theo1:lamda}
\end{enumerate}
\end{theo}

\begin{proof}
Assume that $u$ is not always positive in $\Omega$, that is, there exists a point $x\in\Omega_{-}$ such that $u(x)\leq 0$.
The strong maximum principle ensures that $u$ cannot be zero at an interior minimum, i.e., there exists a point $x\in\Omega_{-}$ such that $u(x)<0$ (the case that $u\equiv 0$ in $\Omega_{-}$ is generally allowed, but this case is also ruled out due to Assumption \ref{theo1:f}; note that $(u\geq)$\,$\underline{u}>0$ in $\Omega_{-}^{\circ}\backslash\Omega_{-}$).
In other words, there exists a nonempty subdomain $\Omega_{-}'\subset\Omega_{-}$ such that $u<0$ in $\Omega_{-}'$ and $u=0$ on $\partial\Omega_{-}'$.
Therefore, the restricted function $v:=-u|_{\Omega_{-}'}$ can be regarded as a solution to
\begin{align*}
\left\{\begin{array}{l l}
-Lv(x)=f_{-}(v(x))\left(:=-f(-v(x))\right) &x\in\Omega_{-}',\\
v(x)>0 &x\in\Omega_{-}',\\
v(x)=0 &x\in\partial\Omega_{-}'.\\
\end{array}\right.
\end{align*}
Since $f_{-}(v(\cdot))\geq 0,\ f_{-}(v(\cdot))\not\equiv 0$, and $f_{-}(v(\cdot))v(\cdot)^{-1}<\lambda_{1}(\hat{\Omega})(<\infty)$ in $\Omega_{-}'$ due to Assumptions \ref{theo1:f2} and \ref{theo1:lamda},
it follows from Lemma \ref{main/lem} that
\begin{align*}
\displaystyle \left(\lambda_{1}(\hat{\Omega})>\sup_{x\in\Omega_{-}'}e(-u(x))=\right)\sup_{x\in\Omega_{-}'}\frac{f_{-}(v(x))}{v(x)}\geq\lambda_{1}\left(\Omega_{-}'\right),
\end{align*}
where $\lambda_{1}\left(\Omega_{-}'\right)$ is the first eigenvalue of \eqref{weak/eig/pro} with the notational replacement $\Omega=\Omega_{-}'$.
Since the inclusion $\Omega_{-}'\subset\hat{\Omega}$ ensures that all functions in $H_{0}^{1}\left(\Omega_{-}'\right)$ can be regarded as functions in $H_{0}^{1}(\hat{\Omega})$ by considering the zero extension outside $\Omega_{-}'$, the inequality $\lambda_{1}\left(\Omega_{-}'\right)\geq\lambda_{1}(\hat{\Omega})$ follows.
Thus, we have the contradiction that $\lambda_{1}(\hat{\Omega})>\lambda_{1}(\hat{\Omega})$.
\end{proof}
\begin{rem}\label{rem:theo1:1}
Assumption $\ref{theo1:boundary}$ in Theorem $\ref{theo1}$ always holds when $u$ satisfies the homogeneous Dirichlet boundary condition \eqref{diri}.
\end{rem}
\begin{rem}\label{rem:theo1:2}
We may employ the condition that
$f([0,{\,}\displaystyle \max\{\overline{u}(x)_{\,}:_{\,}x\in\Omega_{-}\}+\varepsilon])\geq 0$
for a positive number $\varepsilon$, instead of Assumption $\ref{theo1:f}$.
Indeed, since $\overline{u}$ is a continuous function over $\Omega$,
there exists a domain $\Omega_{-}^{\circ}\supset\Omega_{-}$
such that
$\displaystyle \max\{\overline{u}(x) : x\in\Omega_{-}^{\circ}\}\leq\max\{\overline{u}(x) : x\in\Omega_{-}\}+\varepsilon$ for any $\varepsilon>0$.
\end{rem}
The following corollary is convenient to prove the positiveness of a solution $u$ that is proven to exist in a ball centered around an approximation $\hat{u}$.
Indeed, in Section \ref{sec:example}, we present numerical examples where the positiveness of solutions to \eqref{eq}, that are in balls centered around their approximations, is verified.
\begin{coro}\label{coro1:theo1}
Let $u\in C^{2}\left(\Omega\right)\cap C(\overline{\Omega})$ be a function satisfying \eqref{eq} $($a boundary value condition is not necessary$)$
such that $\left|u(x)-\hat{u}(x)\right|\leq r$ for all $x\in \Omega$, for given $\hat{u}\in C(\overline{\Omega})$ and $r>0$.
Function $u$ is always positive in $\Omega$, if the following conditions are satisfied:
\begin{enumerate}
\item$\hat{u}-r$ is positive in a nonempty subdomain $\Omega_{+}\subset\Omega\ (\Omega_{+}\neq\Omega) ;$\label{coro1:1}
\item$\partial\Omega_{-}\cap\partial\Omega=\emptyset$, or $u=0$ {\rm on} $\partial\Omega_{-}\cap\partial\Omega$~if~$\partial\Omega_{-}\cap\partial\Omega\neq\emptyset$, where $\Omega_{-}$ is the interior of $\Omega\backslash\Omega_{+} ;$\label{coro1:2}
\item$f\left([0,2r+\varepsilon]\right)\geq 0$ for a positive number $\varepsilon ;$\label{coro1:3}
\item$f\left([m-r,_{\,}0]\right)\leq 0$ and $f\left([m-r,_{\,}0)\right)<0$, where $m:=\displaystyle \min\{\hat{u}(x)_{\,}:_{\,}x\in\Omega_{-}\} ;$\label{coro1:4}
\item For a domain $\hat{\Omega}$ such that $\Omega_{-}\subset\hat{\Omega}\subset\Omega,\ e([0,-m+r])<\lambda_{1}(\hat{\Omega})$, where $e(x):=f\left(x\right)x^{-1}$, and $\lambda_{1}(\hat{\Omega})$ is the first eigenvalue of the problem \eqref{weak/eig/pro} with the notational replacement $\Omega=\hat{\Omega}$.\label{coro1:5}
\end{enumerate}
\end{coro}
\begin{rem}
Assumption $\ref{coro1:3}$ in Corollary $\ref{coro1:theo1}$ is simplified owing to the discussion in Remark $\ref{rem:theo1:2}$.
\end{rem}

%% file: example.tex
\section{Numerical example}\label{sec:example}
In this section, we present numerical examples in which the positiveness of solutions to \eqref{eq} is verified.
All computations were carried out on a computer with Intel Xeon E7-4830 2.20 GHz$\times$40 processors, 2 TB RAM, CentOS 6.6, and MATLAB 2012b.
All rounding errors were strictly estimated using the toolboxes of INTLAB version 9 \cite{rump1999book} and KV library version 0.4.36 \cite{kashiwagikv} for verified numerical computations.
Therefore, the accuracy of all results was mathematically guaranteed.
Hereafter, $\overline{B}\left(x,r;\ \|\cdot\|\right)$ denotes the closed ball whose center is $x$, and whose radius is $r\geq 0$ in the sense of the norm $\|\cdot\|$.

For the first example, 
we select the case in which
$L=\Delta$ (the usual Laplace operator), $f(u(\cdot))=u(\cdot)^{p}$ and $F(u)=u^p$ ($p=3, 5$), and $\Omega=(0,1)^{2} \subset \mathbb{R}$, i.e., we consider the problem of finding solutions to
\begin{subnumcases}
{\label{problem1}}
-\Delta u=u^p &in $\Omega$,\label{eq1}\\
u>0 &in $\Omega$,\label{positive1}\\
u=0 &on $\partial\Omega$.\label{diri1}
\end{subnumcases}
We computed approximate solutions $\hat{u}$ to
\begin{align}
\left\{\begin{array}{l l}
-\Delta u=\left|u\right|^{p-1}u &\mathrm{in~} \Omega,\\
u=0 &\mathrm{on~} \partial\Omega,
\end{array}\right.\label{absproblem}
\end{align}
which is displayed in Fig.~\ref{fig1}, using a Legendre polynomial basis, i.e., we constructed $\hat{u}$ as
\begin{align}
\displaystyle \hat{u}=\sum_{i,j=1}^{N}u_{i,j}\phi_{i}\phi_{j},~~u_{i,j}\in \mathbb{R},
\end{align}
where each $\phi_{i}$ is defined by
\begin{align}
\displaystyle \phi_{n}(x)=\frac{1}{n(n+1)}x(1-x)\frac{dP_{n}}{dx}(x),~~n=1,2,3,\cdots
\end{align}
with the Legendre polynomials $P_{n}$ defined by
\begin{align}
P_{n}=\displaystyle \frac{(-1)^{n}}{n!}\left(\frac{d}{dx}\right)^{n}x^{n}(1-x)^{n},~~n=0,1,2,\cdots.
\end{align}
We proved the existence of solutions $u$ to \eqref{absproblem} in an $H_{0}^{1}$-ball $\overline{B}(\hat{u},r_{1};\,\|\nabla \cdot\|_{L^2(\Omega)})$ and an $L^{\infty}$-ball $\overline{B}(\hat{u},r_{2};\,\|\cdot\|_{L^{\infty}(\Omega)})$ both centered around the approximations $\hat{u}$, using 
Theorem \ref{plum2001} \cite{plum2001computer} combined with the method in \cite{tanaka2014verified, liu2015framework} (see Section \ref{verificationtheory} for details).
We then verified the positiveness of the verified solutions using Corollary \ref{coro1:theo1}.
For the problem \eqref{absproblem}, Assumptions \ref{coro1:2}, \ref{coro1:3}, and \ref{coro1:4} in Corollary \ref{coro1:theo1} are always satisfied.
Therefore, it is only necessary to confirm Assumptions \ref{coro1:1} and \ref{coro1:5}, where $e(u)=u^{p-1}$ and $\hat{\Omega}=\Omega$.
Note that the verified solutions have the regularity to be in $C^{2}\left(\Omega\right)\cap C\left(\overline{\Omega}\right)$, regardless of the regularity of the corresponding approximations $\hat{u}$.
Indeed, for each $h\in L^{2}\left(\Omega\right)$, the problem
\begin{align*}
\left\{\begin{array}{l l}
-\Delta u=h &\mathrm{in}\ \Omega,\\
u=0 &\mathrm{on}\ \partial\Omega\\
\end{array}\right.
\end{align*}
has a unique solution $u\in H^{2}(\Omega)$, such as when $\Omega$ is a bounded convex domain with a piecewise $C^{2}$ boundary $($see, e.g.,  \cite[Section 3.3]{grisvard2011elliptic}).
Therefore, the so-called bootstrap argument ensures that a weak solution $u\in H_{0}^{1}\left(\Omega\right)$ to \eqref{absproblem} on such a domain $\Omega$, is in $C^{\infty}\left(\Omega\right)(\subset C^{2}\left(\Omega\right))$.
Table.\,\ref{tab1} presents the verification result, which ensures the positiveness of the verified solutions to \eqref{absproblem} centered around $\hat{u}$, owing to the condition that
$e([0,-m_{\Omega_{-}}+r])\displaystyle \leq(-m_{\Omega}+r)^{p-1}<\lambda_{1}(\Omega)$,
where we denote $m_{\Omega_{-}}=\min\{\hat{u}(x)_{\,}:_{\,}x\in\Omega_{-}\}$ and $m_{\Omega}=\min\{\hat{u}(x)_{\,}:_{\,}x\in\Omega\}$.

 \newcommand{\sizee}{0.5\hsize}
 \newcommand{\vminus}{\vspace{0mm}}
 \begin{figure}[H]
 \begin{minipage}{\sizee}
	 \begin{center}
	   \vminus\includegraphics[height=55 mm]{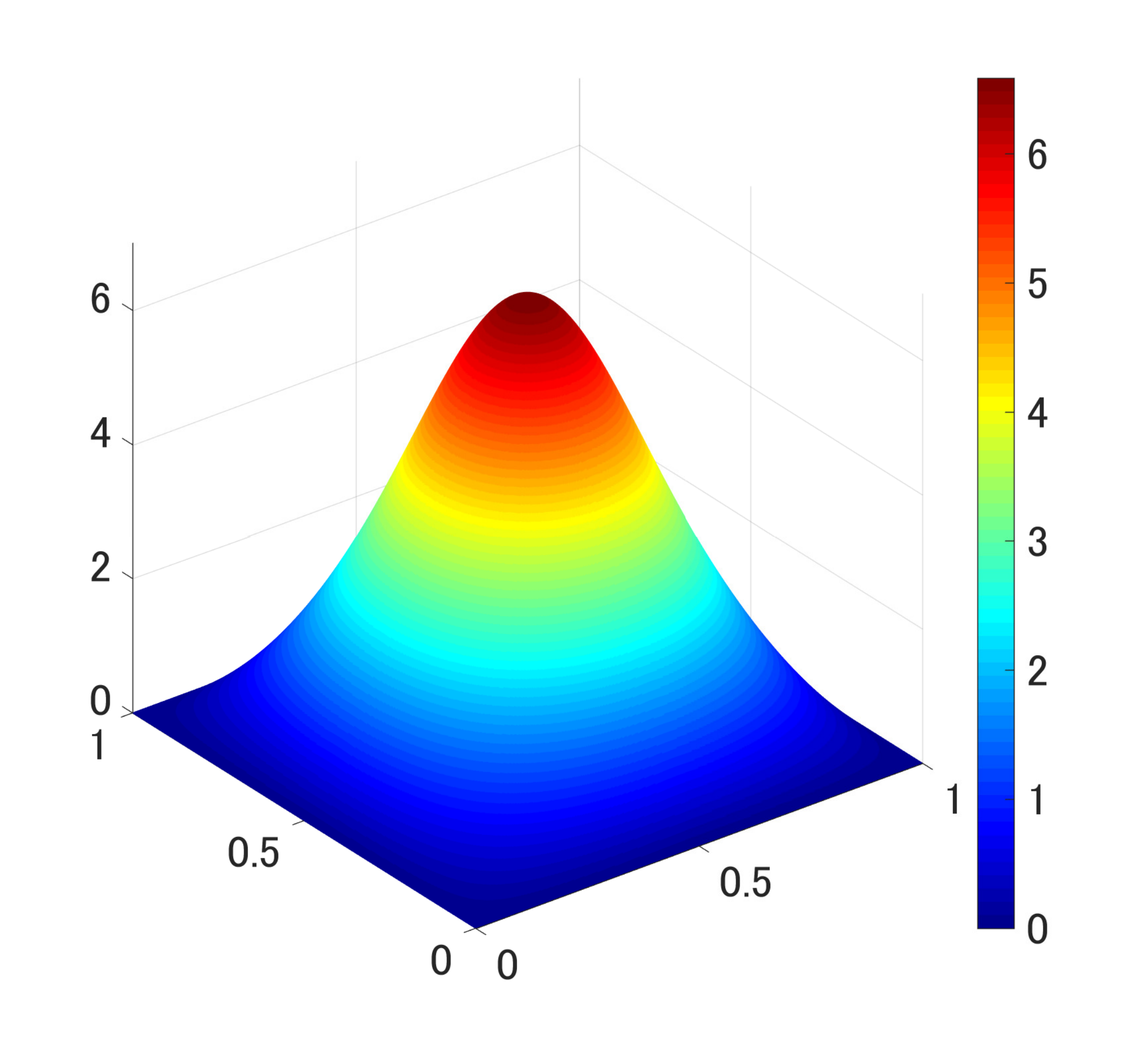}\\
	   \footnotesize$p=3$,~~$\displaystyle \max_{x\in\Omega}\hat{u}(x)\approx 6.6232$
	 \end{center}
	 ~
 \end{minipage}
 \begin{minipage}{\sizee}
	 \begin{center}
	   \vminus\includegraphics[height=55 mm]{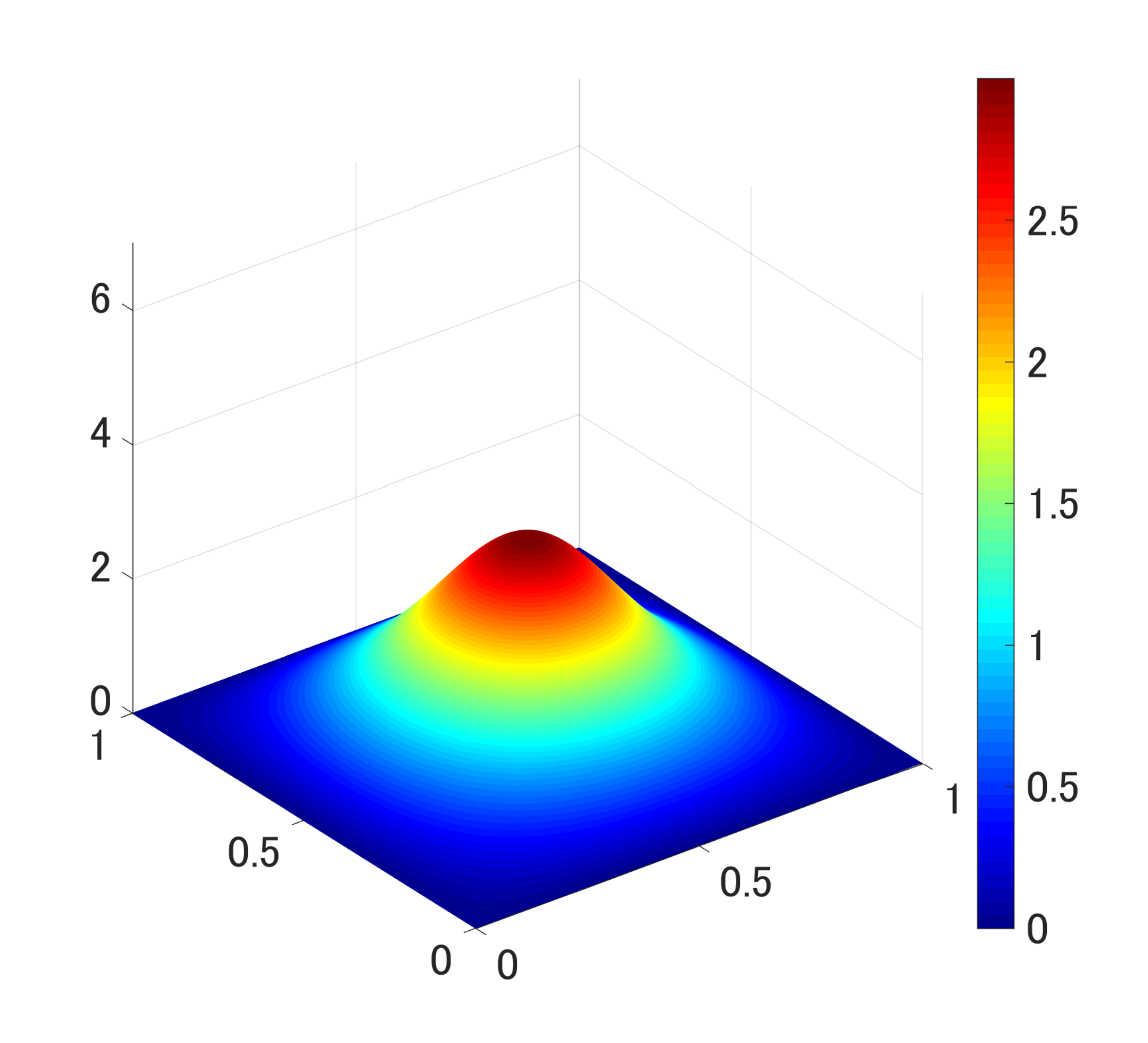}\\
	   \footnotesize$p=5$,~~$\displaystyle \max_{x\in\Omega}\hat{u}(x)\approx 3.1721$
	 \end{center}
	 ~
 \end{minipage}
 \caption{Approximate solutions to \eqref{absproblem} on $\Omega=(0,1)^{2}$ for $p=3,5$.}
 \label{fig1}
 \end{figure} 

 \begin{table}[h]
  \caption{Verification results for \eqref{problem1} on $\Omega=(0,1)^{2}$, where $p=3$ and $5$.}
  \label{tab1}
  \begin{center}
   \renewcommand\arraystretch{1.3}
   \footnotesize
   \begin{tabular}{c|cccccc}
    \hline
    $p$&
	$N$&
    $r_{1}$&
    $r_{2}$&
	$m_{\Omega}$&
	$(-m_{\Omega}+r)^{p-1}\leq$&
    $\lambda_{1}(\Omega)$\\
    \hline
    \hline
    3&
	150&
	  6.636469152e-13&
	  4.363745213e-12&
	  0&
	  1.904227228e-23&
    $(19.7 \leq)~2\pi^{2}$\\
    5&
	150&
	  7.088374332e-13&
	  1.724519836e-10&
	  0&
	  8.844489601e-40&
    $''$\\
    \hline
   \end{tabular}
  \end{center}
 \end{table}

In the next example, we consider the stationary problem of the Allen-Cahn equation:
\begin{subnumcases}
{\label{problem3}}
-\varepsilon^{2}\Delta u=u-u^3 &in $\Omega$,\label{eq3}\\
u>0 &in $\Omega$,\label{positive3}\\
u=0 &on $\partial\Omega$,\label{diri3}
\end{subnumcases}
where $\varepsilon>0$.
We again constructed approximate solutions $\hat{u}$ to \eqref{eq3} with \eqref{diri3}, i.e., to
\begin{align}
\left\{\begin{array}{l l}
-\Delta u=\varepsilon^{-2}(u-u^3) &\mathrm{in~} \Omega,\\
u=0 &\mathrm{on~} \partial\Omega,
\end{array}\right.\label{allenabsproblem}
\end{align}
using a Legendre polynomial basis.
These solutions are displayed in Fig.~\ref{fig3}.
Using Theorem \ref{plum2001}, we again verified the existence of solutions to \eqref{allenabsproblem} in the balls $\overline{B}(\hat{u},r_{1};\,\|\nabla \cdot\|_{L^2(\Omega)})$ and $\overline{B}(\hat{u},r_{2};\,\|\cdot\|_{L^{\infty}(\Omega)})$, which are also $C^2$-regular.
For the problem \eqref{allenabsproblem}, Assumption \ref{coro1:2} in Corollary \ref{coro1:theo1} is always satisfied as well.
Therefore, it is necessary to confirm Assumptions \ref{coro1:1}, \ref{coro1:3}, \ref{coro1:4}, and \ref{coro1:5}, where $e(u)=\varepsilon^{-2}(1-u^2)$.
Note that Assumption \ref{coro1:3} is satisfied if $2r_2<1$,
and Assumption \ref{coro1:4} is satisfied if $-m_{\Omega_{-}}+r_2(<-m_{\Omega}+r_2)<1$.
We present the verification results for $\varepsilon=0.1$, $0.05$, and $0.025$ in Table \ref{tab3}, which ensure the positiveness of the verified solutions to \eqref{allenabsproblem} centered around $\hat{u}$, owing to the condition that $e([0,-m_{\Omega_{-}}+r_2])\leq\varepsilon^{-2}<\lambda_{1}(\hat{\Omega})$, where we set $\hat{\Omega}=(0,1)^2\backslash[0.009765625,0.990234375]^{2}$ and proved $\Omega_{-}\subset\hat{\Omega}$ in all the cases.
The upper and lower bounds for the first eigenvalue $\lambda_{1}(\hat{\Omega})$ were numerically computed with verification using the method in \cite{liu2013verified,liu2015framework} with a piecewise linear finite element basis.

 \renewcommand{\sizee}{0.33\hsize}
 \begin{figure}[H]
 \begin{minipage}{\sizee}
	 \begin{center}
	   \vminus\includegraphics[height=40 mm]{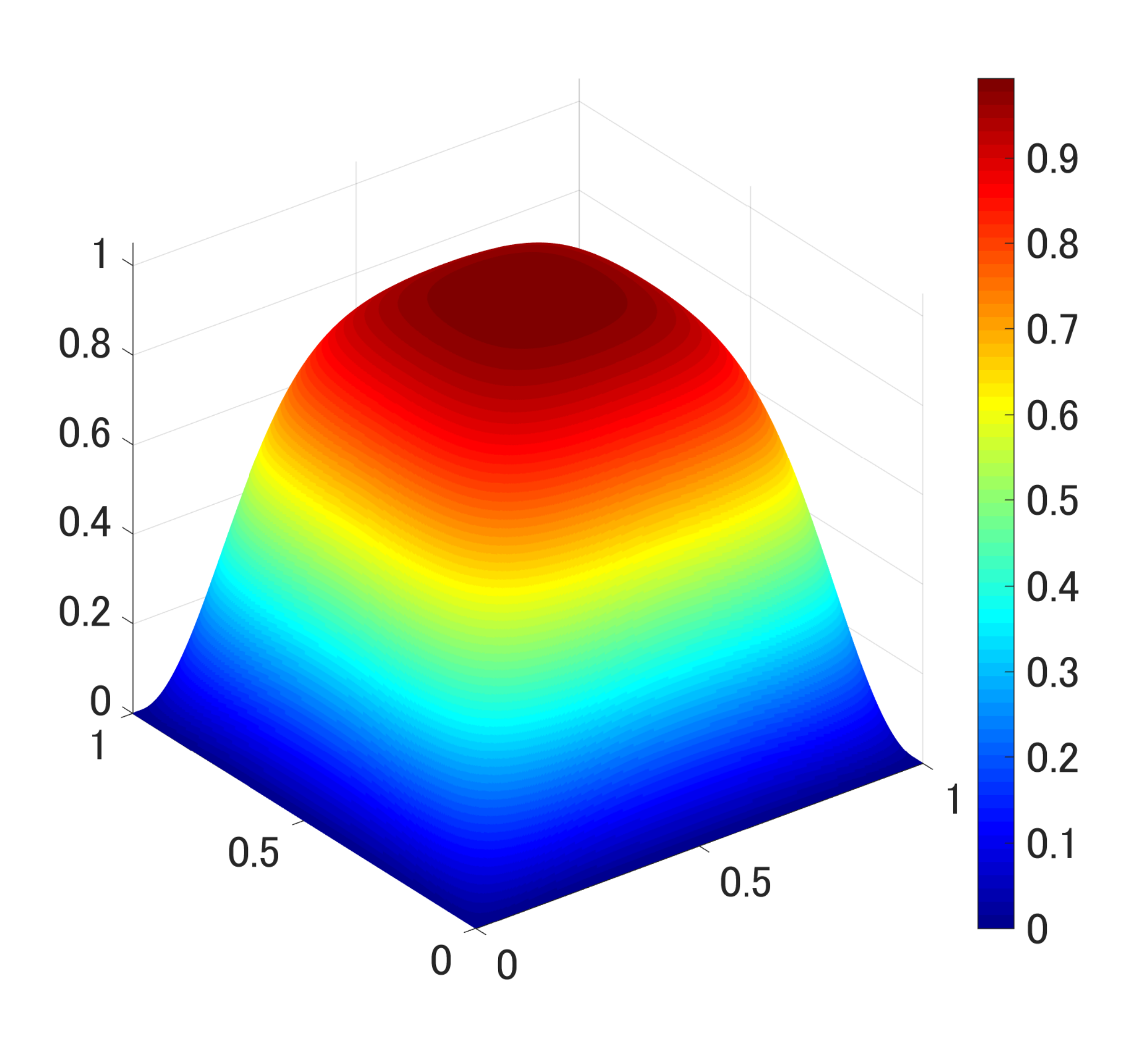}\\
	   \footnotesize{$\varepsilon=0.1$}
	 \end{center}
	 ~
 \end{minipage}
 \begin{minipage}{\sizee}
	 \begin{center}
	   \vminus\includegraphics[height=40 mm]{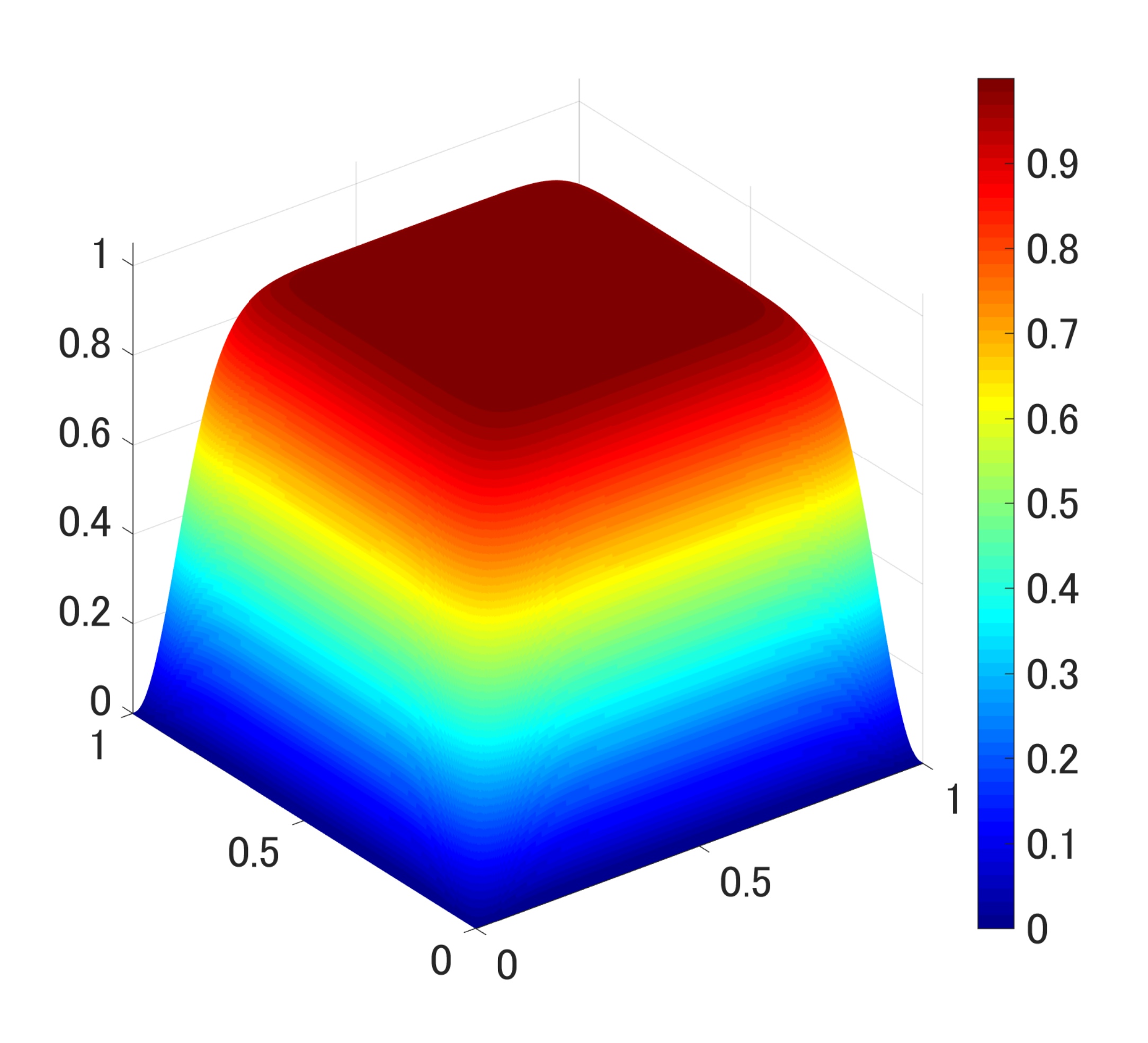}\\
	   \footnotesize{$\varepsilon=0.05$}
	 \end{center}
	 ~
 \end{minipage}
 \begin{minipage}{\sizee}
	 \begin{center}
	   \vminus\includegraphics[height=40 mm]{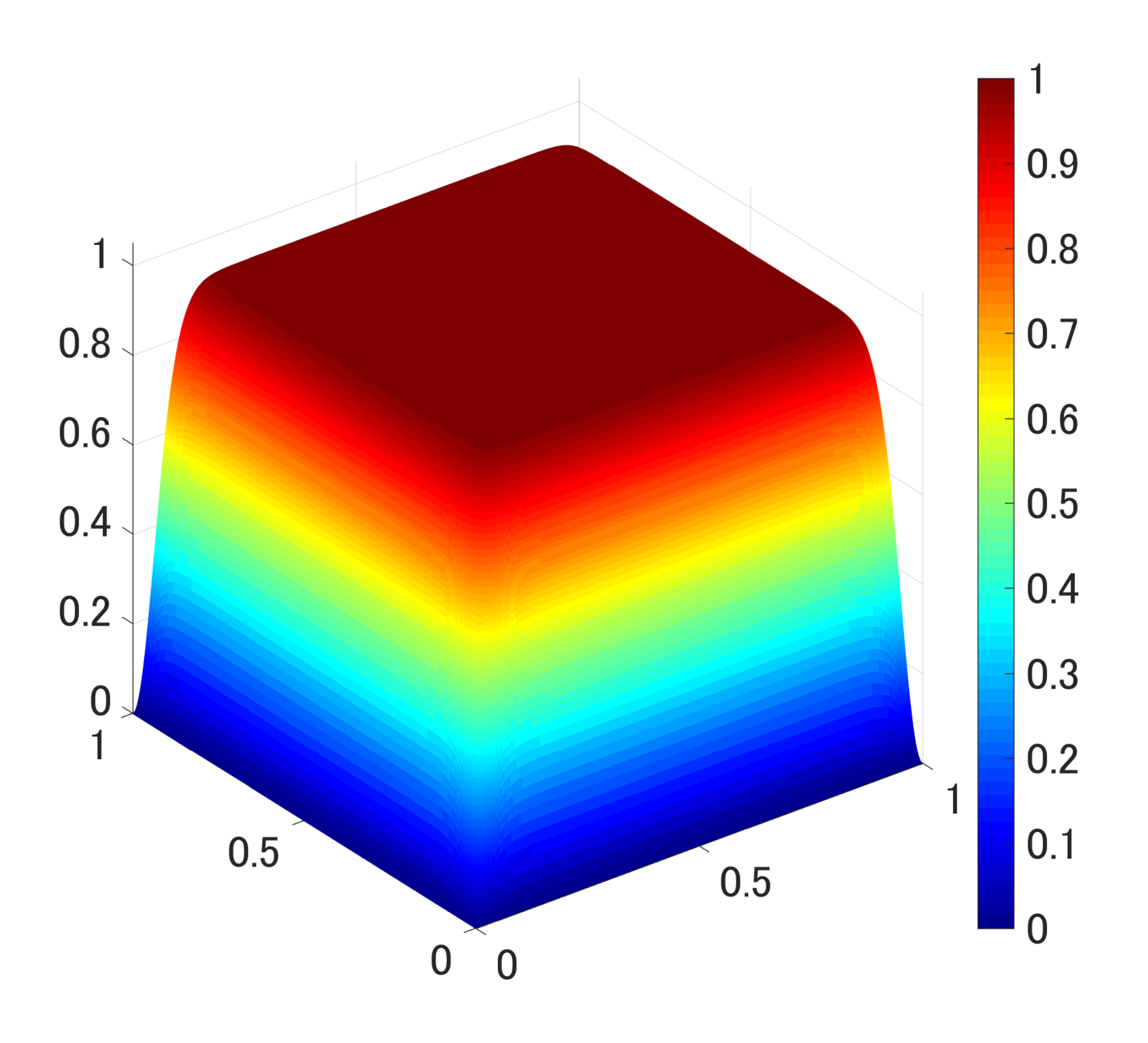}\\
	   \footnotesize{$\varepsilon=0.025$}
	 \end{center}
	 ~
 \end{minipage} 
 \caption{Approximate solutions to \eqref{allenabsproblem} on $\Omega=(0,1)^{2}$.}
 \label{fig3}
 \end{figure} 

 \begin{table}[h]
  \caption{Verification results for \eqref{problem3} on $\Omega=(0,1)^{2}$.}
  \label{tab3}
  \begin{center}
   \renewcommand\arraystretch{1.3}
   \footnotesize
   \begin{tabular}{l|cccccc}
    \hline
    $\varepsilon$&
	$N$&
    $r_{1}$&
    $r_{2}$&
	$m_{\Omega}\geq$&
    $\varepsilon^{-2}$&
    $\lambda_{1} (\hat{\Omega}) \in$\\
    \hline
    \hline
    0.1&
	60&
	  5.820766091e-15&
	  4.685104029e-14&
	 -2.543817144e-03&
	  1.0e+02&
      $[0.95852,1.00312]$e+05 \\
    0.05&
	60&
	  9.584445593e-10&
	  3.937471389e-08&
	 -9.935337351e-03&
	  4.0e+02&
      $''$\\
    0.025&
	80&
	  1.385525794e-08&
	  2.611277945e-06&
	 -3.056907364e-02&
	  1.6e+03&
      $''$\\

	\hline
   \end{tabular}
  \end{center}
 \end{table}

%% file: appendix_1.tex
\section{Verification theorem for elliptic problems}\label{verificationtheory}
In this section, we apply the method summarized in \cite{plum1992explicit, plum2001computer, plum2008} to a verified numerical computation for solutions to \eqref{eq} with \eqref{diri}.
Hereafter, we denote $V=H_{0}^{1}\left(\Omega\right)$.
We assume that $F$ defined in \eqref{kataiF} is Fr\'echet differentiable, and there exists $q\in L^{\infty}(\Omega)$ such that
\begin{align}
F_{\hat{u}}'u=qu~~~{\rm for}~u\in V,\label{Fdivcond}
\end{align}
where $F_{\hat{u}}'$ is the Fr\'echet derivative of $F$ at $\hat{u}\in V$, i.e., we may regard $F_{\hat{u}}'$ as an $L^{\infty}$ function.
Indeed, when we select $F$ as in the examples of Section \ref{sec:example}, this is true.
We endow $V$ with inner product
\begin{align*}
\left(\cdot,\cdot\right)_{V}=\left(\nabla\cdot,\nabla\cdot\right)_{L^{2}(\Omega)}+\tau\left(\cdot,\cdot\right)_{L^{2}(\Omega)}
\end{align*}
and norm $\left\|\cdot\right\|_{V}:=\sqrt{\left(\cdot,\cdot\right)_{V}}$,
where $\tau$ is a nonnegative number chosen as
\begin{align}
\tau>-F_{\hat{u}}'(x)\ (\mathrm{a}.\mathrm{e}.\ x\in\Omega).\label{select:tau}
\end{align}
Note that, since the norm $\left\|\cdot\right\|_{V}$ monotonically increases with $\tau$, the usual norm $\left\|\nabla\cdot\right\|_{L^{2}(\Omega)}$ is bounded by $\left\|\cdot\right\|_{V}$ for any $\tau$; therefore, $\overline{B}(\hat{u},r_{1};$\,$\|\nabla\cdot\|_{L^{2}(\Omega)})\subset\overline{B}(\hat{u},r_{1};$\,$\|\cdot\|_{V})$ for any $r_{1}\geq 0$ and $\tau\geq 0$.
Moreover, we denote $V^{*}=H^{-1}\left(\Omega\right)$($:=$(dual of $V$)) with the usual sup-norm.
The $L^{2}$-inner product and the $L^{2}$-norm are simply denoted by $\left(\cdot,\cdot\right)$ and $\left\|\cdot\right\|$, respectively, if no confusion arises.
By defining $\mathcal{F}:V\rightarrow V^{*}$ as
\begin{align*}
\left\langle\mathcal{F}(u),v\right\rangle:=\left(\nabla u,\nabla v\right)-\left(F\left(u\right),v\right)~~~{\rm for}~u,v\in V,
\end{align*}
we first re-write \eqref{eq} with \eqref{diri} as
\begin{align}
\mathcal{F}(u)=0~{\rm in}~V^{*},\label{gFproblem}
\end{align}
and discuss the verified numerical computation for \eqref{gFproblem}.
In other words, we first consider the existence of a weak solution to \eqref{eq} with \eqref{diri} (a solution to \eqref{gFproblem} in $V$), and then we discuss its $H^{2}$-regularity if necessary.
The norm bound for the embedding $V\hookrightarrow L^{p}\left(\Omega\right)$ is denoted by $C_{p}$, i.e., $C_{p}$ is a positive number that satisfies
\begin{align}
\left\|u\right\|_{L^{p}(\Omega)}\leq C_{p}\left\|u\right\|_{V}~~~{\rm for~all}~u\in V.\label{embedding}
\end{align}
Since an explicit upper bound for $C_{p}$ is important for the verification theory,
formulas that give such an upper bound are provided in \ref{sec:embedding}.

\subsection{$H_{0}^{1}$ error estimation}
We use the following verification theorem for obtaining $H_{0}^{1}$ error estimations for solutions to \eqref{gFproblem}.

\begin{theo}[\cite{plum2001computer}]\label{plum2001}
Let $\mathcal{F}:V\rightarrow V^{*}$ be a Fr\'echet differentiable operator.
Suppose that $\hat{u}\in V$, and that there exist $\delta>0,\ K>0$, and a non-decreasing function $g$ satisfying
\begin{align}
&\left\|\mathcal{F}\left(\hat{u}\right)\right\|_{V^{*}}\leq\delta,\label{zansa}\\
&\left\|u\right\|_{V}\leq K\left\|\mathcal{F}_{\hat{u}}'u\right\|_{V^{*}}~~~{\rm for~all}~u\in V,\label{inverse}\\
&\left\|\mathcal{F}_{\hat{u}+u}'-\mathcal{F}_{\hat{u}}'\right\|_{\mathcal{B}(V,V^{*})}\leq g\left(\left\|u\right\|_{V}\right)~~~{\rm for~all}~u\in V,\label{lip}
\shortintertext{and}
&g(t)\rightarrow 0~~{\rm as}~~t\rightarrow 0\label{gconv}.
\end{align}
Moreover, suppose that some $\alpha>0$ exists such that
\begin{align*}
\displaystyle \delta\leq\frac{\alpha}{K}-G\left(\alpha\right)~~and~~Kg\left(\alpha\right)<1,
\end{align*}
where $G(t):=\displaystyle \int_{0}^{t}g(s)ds$.
Then, there exists a solution $u\in V$ to the equation $\mathcal{F}(u)=0$ satisfying
\begin{align}
\left\|u-\hat{u}\right\|_{V}\leq\alpha.\label{al}
\end{align}
Furthermore, the solution is unique under the side condition \eqref{al}.
\end{theo}

Note that, the Fr\'echet derivative $\mathcal{F}_{\hat{u}}'$ of $\mathcal{F}$ at $\hat{u}\in V$ is given by
\begin{align*}
\left\langle\mathcal{F}_{\hat{u}}'u,v\right\rangle=\left(\nabla u,\nabla v\right)-\left(F_{\hat{u}}'u,v\right)~~~{\rm for}~u,v\in V.
\end{align*}
\subsubsection*{Residual bound $\delta$}
For $\hat{u}\in V$ that satisfies $\Delta\hat{u}\in L^{2}\left(\Omega\right)$, the residual bound $\delta$ is computed as
\begin{align*}
C_{2}\left\|\Delta\hat{u}+F(\hat{u})\right\|_{L^{2}\left(\Omega\right)}\left(\leq\delta\right);
\end{align*}
the $L^{2}$-norm can be computed by a numerical integration method with verification.

\subsubsection*{Bound $K$ for the operator norm of $\mathcal{F}_{\hat{u}}^{\prime-1}$}
We compute a bound $K$ for the operator norm of $\mathcal{F}_{\hat{u}}^{\prime-1}$ by the following theorem,
proving simultaneously that this inverse operator exists and is defined on the whole of $V^{*}$.
\begin{theo}[\cite{plum2009computer}]\label{invtheo}
Let $\Phi:V\rightarrow V^{*}$ be the canonical isometric isomorphism, i.e., $\Phi$ is given by
\begin{align*}
\left\langle\Phi u,v\right\rangle:=\left(u,v\right)_{V}~~~{\rm for}~u,v\in V.
\end{align*}
If the point spectrum of $\Phi^{-1}\mathcal{F}_{\hat{u}}'$~$($denoted by $\sigma_{p}(\Phi^{-1}\mathcal{F}_{\hat{u}}') )$ does not contain zero, then the inverse of $\mathcal{F}_{\hat{u}}'$ exists and
\begin{align}
\left\|\mathcal{F}_{\hat{u}}^{\prime-1}\right\|_{B(V^{*},V)}\leq\mu_{0}^{-1},\label{Ktheo}
\end{align}
where
\begin{align}
\displaystyle \mu_{0}=\min\left\{|\mu|\ :\ \mu\in\sigma_{p}\left(\Phi^{-1}\mathcal{F}_{\hat{u}}'\right)\cup\{1\}\right\}.\label{mu0}
\end{align}
\end{theo}
The eigenvalue problem $\Phi^{-1}\mathcal{F}_{\hat{u}}'u=\mu u$ in $V$ is equivalent to
\begin{align*}
\left(\nabla u,\nabla v\right)-\left(F_{\hat{u}}'u,v\right)=\mu\left(u,v\right)_{V}~~{\rm for~all}~v\in V.
\end{align*}
Since $\mu=1$ is already known to be in $\sigma\left(\Phi^{-1}\mathcal{F}_{\hat{u}}'\right)$, it suffices to look for eigenvalues $\mu\neq 1$.
By setting $\lambda=(1-\mu)^{-1}$, we further transform this eigenvalue problem into
\begin{align}
{\rm Find}~u\in V~{\rm and}~\lambda\in \mathbb{R}~{\rm s.t.}~\left(u,v\right)_{V}=\lambda\left((\tau+F_{\hat{u}}')u,v\right)~~{\rm for~all}~v\in V.\label{eiglam}
\end{align}
Owing to \eqref{select:tau}, \eqref{eiglam} is a regular eigenvalue problem, the spectrum of which consists of a sequence $\{\lambda_{k}\}_{k=1}^{\infty}$ of eigenvalues converging to $+\infty$.
In order to compute $K$ on the basis of Theorem \ref{invtheo},
we concretely enclose the eigenvalue $\lambda$ of \eqref{eiglam} that minimizes the corresponding absolute value of $|\mu|\left(=|1-\lambda^{-1}|\right)$, by considering the following approximate eigenvalue problem
\begin{align}
{\rm Find}~u\in V_{N}~{\rm and}~\lambda^{N}\in \mathbb{R}~{\rm s.t.}~\left(u_{N},v_{N}\right)_{V}=\lambda^{N}\left((\tau+F_{\hat{u}}')u_{N},v_{N}\right)~~{\rm for~all}~v_{N}\in V_{N},\label{applam}
\end{align}
where $V_{N}$ is a finite-dimensional subspace of $V$.

We estimate the error between the $k$th eigenvalue $\lambda_{k}$ of \eqref{eiglam} and the $k$th eigenvalue $\lambda_{k}^{N}$ of \eqref{applam},
by considering the weak formulation of the Poisson equation
\begin{align}
\left(u,v\right)_{V}=\left(h,v\right)_{L^2(\Omega)}~~~{\rm for~all}~v\in V\label{poisson}
\end{align}
for given $h\in L^{2}\left(\Omega\right)$; it is well known that this equation has a unique solution $u\in V$ for each $h\in L^{2}\left(\Omega\right)$.
Moreover, we introduce the orthogonal projection $P_{N}^{\tau}:V\rightarrow V_{N}$ defined by
\begin{align*}
\left(P_{N}^{\tau}u-u,v_{N}\right)_{V}=0~~~{\rm for~all}~u\in V{\rm~and~}v_{N}\in V_{N}.
\end{align*}
The following theorem enables us to estimate the error between $\lambda_{k}$ and $\lambda_{k}^{N}$.
\begin{theo}[\cite{tanaka2014verified, liu2015framework}]\label{eigtheo}
Suppose that there exists a positive number $C_{N}^{\tau}$ such that
\begin{align}
\left\|u_{h}-P_{N}^{\tau}u_{h}\right\|_{V}\leq C_{N}^{\tau}\left\|h\right\|_{L^{2}(\Omega)}\label{CN}
\end{align}
for any $h\in L^{2}\left(\Omega\right)$ and the corresponding solution $u_{h}\in V$ to \eqref{poisson}.
Then,
\begin{align*}
\displaystyle \frac{\lambda_{k}^{N}}{\lambda_{k}^{N}\left(C_{N}^{\tau}\right)^{2}\|\tau+F_{\hat{u}}'\|_{L^{\infty}(\Omega)}+1}\leq\lambda_{k}\leq\lambda_{k}^{N},
\end{align*}
where $F_{\hat{u}}'$ is regarded as an $L^{\infty}$ function owing to \eqref{Fdivcond}.
\end{theo}
The inequality on the right is well known as a Rayleigh-Ritz bound, which is derived from the min-max principle:
\begin{align*}
\displaystyle \lambda_{k}=\min_{H_{k}\subset V}\left(\max_{v\in H_{k}\backslash\{0\}}\frac{\left\|v\right\|_{V}^{2}}{\left\|av\right\|_{L^{2}(\Omega)}^{2}}\right)\leq\lambda_{k}^{N},
\end{align*}
where we set $a=\sqrt{\tau+F_{\hat{u}}'}$ and the minimum is taken over all $k$-dimensional subspaces $H_{k}$ of $V$.
Moreover, proofs of the inequality on the left can be found in \cite{tanaka2014verified, liu2015framework}.
Assuming the $H^{2}$-regularity of solutions to \eqref{poisson} (e.g., when $\Omega$ is convex {\rm \cite[Section 3.3]{grisvard2011elliptic}}), \cite[Theorem 4]{tanaka2014verified} ensures the left inequality.
A more general statement, that does not require the $H^{2}$-regularity, can be found in \cite[Theorem 2.1]{liu2015framework}.

\begin{rem}
When the $H^{2}$-regularity of solutions to \eqref{poisson} is confirmed a priori, e.g., when $\Omega$ is convex {\rm \cite[Section 3.3]{grisvard2011elliptic}}, \eqref{CN} can be replaced by
\begin{align}
\left\|u-P_{N}^{\tau}u\right\|_{V}\leq C_{N}^{\tau}\left\|-\Delta u+\tau u\right\|_{L^{2}(\Omega)}~~~{\rm for~all}~u\in H^{2}(\Omega)\cap V.
\end{align}
We can compute an explicit value of $C_{N}^{0}$ with $\tau=0$, for $V_{N}$ spanned by a Legendre basis $\left\{\phi_{i}\right\}_{i=1}^{N}$ using {\rm \cite[Theorem 2.3]{kimura1999on}}, and may employ $C_{N}^{\tau}=C_{N}^{0}\sqrt{1+\tau(C_{N}^{0})^{2}}$, since
\begin{align*}
\left\|u-P_{N}^{\tau}u\right\|_{\tau}^{2}&\leq\left\|u-P_{N}^{0}u\right\|_{\tau}^{2}=\left\|\nabla\left(u-P_{N}^{0}u\right)\right\|^{2}+\tau\left\|u-P_{N}^{0}u\right\|^{2}\\
&\leq\left\|\nabla\left(u-P_{N}^{0}u\right)\right\|^{2}+\tau(C_{N}^{0})^{2}\left\|\nabla\left(u-P_{N}^{0}u\right)\right\|^{2}=\left(1+\tau(C_{N}^{0})^{2}\right)\left\|\nabla\left(u-P_{N}^{0}u\right)\right\|^{2}\\
&\leq\left(1+\tau(C_{N}^{0})^{2}\right)(C_{N}^{0})^{2}\left\|-\Delta u\right\|^{2}\leq\left(1+\tau(C_{N}^{0})^{2}\right)(C_{N}^{0})^{2}\left\|-\Delta u+\tau u\right\|^{2},
\end{align*}
where the last inequality follows from {\rm \cite[Lemma~1]{tanaka2014verified}}.
\end{rem}

\subsubsection*{Lipschitz bound $\mathcal{F}_{\hat{u}}'$}
A concrete construction of a function $g$ satisfying \eqref{lip} and \eqref{gconv} is important for our verification process.
For the nonlinearity $F(u)=au+bu^{p}$ with $p\geq 2$ and $a,b\in L^{\infty}(\Omega)$ (this form is applicable to the concrete nonlinearities in Section \ref{sec:example}), one can employ
\begin{align*}
g(t)=\left\|b\right\|_{L^{\infty}(\Omega)}p(p-1)C_{p+1}^{3}Kt\left(\left\|\hat{u}\right\|_{L^{p+1}(\Omega)}+C_{p+1}t\right)^{p-2},
\end{align*}
where $\delta$ and $K$ are the respective constants in \eqref{zansa} and \eqref{inverse} for $\hat{u}\in V$.
This selection can be found in \cite[Theorem 3.1]{mckenna2009uniqueness}.
\subsection{$L^{\infty}$ error estimation}
In this subsection, we discuss a method that gives an $L^{\infty}$ error bound for a solution to \eqref{eq} with \eqref{diri} from a known $H_{0}^{1}$ error bound, that is, 
we compute an explicit bound for $\left\|u-\hat{u}\right\|_{L^{\infty}\left(\Omega\right)}$ for a solution $u\in V$ to \eqref{eq} with \eqref{diri} satisfying
\begin{align}
\left\|u-\hat{u}\right\|_{V}\leq\rho\label{h10error}
\end{align}
with $\rho>0$ and $\hat{u}\in V$.
We assume that $\Omega$ is convex and polygonal to obtain such an error estimation; this condition gives the $H^{2}$-regularity of solutions to \eqref{eq} with \eqref{diri} (and therefore, ensures their boundedness) a priori.
More precisely, when $\Omega$ is a convex polygonal domain, a weak solution $u\in V$ to \eqref{poisson} and $h\in L^{2}\left(\Omega\right)$ is $H^{2}$-regular (see, e.g., \cite[Section 3.3]{grisvard2011elliptic}).
A solution $u$ satisfying \eqref{h10error} can be written in the form $ u=\hat{u}+\rho\omega$ with some $\omega\in V,\ \left\|\omega\right\|_{V}\leq 1$.
Moreover, $\omega$ satisfies
\begin{align*}
\left\{\begin{array}{l l}
-\Delta\rho\omega=F\left(\hat{u}+\rho\omega\right)+\Delta\hat{u} &\mathrm{i}\mathrm{n}\ \Omega,\\
\omega=0 &\mathrm{on}\ \partial\Omega,
\end{array}\right.
\end{align*}
and therefore is also $H^{2}$-regular if $\Delta\hat{u}\in L^{2}(\Omega)$.
We then use the following theorem to obtain an $L^{\infty}$ error estimation.
\begin{theo}[\cite{plum1992explicit}]\label{linf}
For all $u\in H^{2}\left(\Omega\right)$,
\begin{align*}
\|u\|_{L^{\infty}(\Omega)}\le c_{0}\|u\|_{L^{2}(\Omega)}+c_{1}\|\nabla u\|_{L^{2}(\Omega)}+c_{2}\|u_{xx}\|_{L^{2}(\Omega)}
\end{align*}
with
\begin{align*}
c_{j}=\displaystyle \frac{\gamma_{j}}{\left|\overline{\Omega}\right|}\left[\max_{x_{0}\in\overline{\Omega}}\int_{\overline{\Omega}}|x-x_{0}|^{2j}dx\right]^{1/2},~(j=0,1,2),
\end{align*}
where $u_{xx}$ denotes the Hesse matrix of $u,\ \left|\overline{\Omega}\right|$ is the measure of $\overline{\Omega}$, and
\begin{align*}
\gamma_{0}=1,~\gamma_{1}=1.1548,~\gamma_{2}=0.22361.
\end{align*}
For $n=3$, other values of $\gamma_{0},~\gamma_{1},$ and $\gamma_{2}$ have to be chosen $($see {\rm \cite{plum1992explicit}}$)$.
\end{theo}
\begin{rem}
The norm of the Hesse matrix of $u$ is precisely defined by
\begin{align*}
\|u_{xx}\|_{L^{2}(\Omega)}=\sqrt{\sum_{i,j=1}^{2}\left\|\frac{\partial^{2}u}{\partial x_{i}\partial x_{j}}\right\|_{L^{2}(\Omega)}^{2}}.
\end{align*}
Moreover, since $\Omega$ is polygonal, $\left\|u_{xx}\right\|_{L^{2}(\Omega)}=\left\|\Delta u\right\|_{L^{2}(\Omega)}$ for all $u\in H^{2}(\Omega)\cap V$~$($see, e.g., {\rm \cite{grisvard2011elliptic}}$)$.
\end{rem}
\begin{rem}
Explicit values of each $c_{j}$ are provided for some special domains $\Omega$ in {\rm \cite{plum1992explicit,plum2001computer}}.
According to these papers, one can choose, for $\Omega=(0,1)^{2}$,
\begin{align*}
c_{0}=\displaystyle \gamma_{0},\ c_{1}=\sqrt{\frac{2}{3}}\gamma_{1},{\rm~and~}c_{2}=\frac{\gamma_{3}}{3}\sqrt{\frac{28}{5}} .
\end{align*}
\end{rem}
Applying Theorem \ref{linf}, we obtain the following corollaries.

\begin{coro}\label{Linfcoro1}
Let $u$ be a solution to \eqref{absproblem} with $p\geq 2$, satisfying \eqref{h10error} for $\hat{u}\in V$ such that $\Delta\hat{u}\in L^{2}\left(\Omega\right)$.
Moreover, let $c_{0},\ c_{1}$, and $c_{2}$ be as in Theorem {\rm \ref{linf}}.
Then,
\begin{align}
&\left\|u-\hat{u}\right\|_{L^{\infty}(\Omega)}\nonumber\\
&\leq c_{0}C_{2}\rho+c_{1}\rho+c_{2}\left(2^{p-\frac{3}{2}}p\rho C_{3}\sqrt{\left\|\hat{u}\right\|_{L^{6(p-1)}\left(\Omega\right)}^{2\left(p-1\right)}+\frac{\rho^{2(p-1)}}{2p-1}C_{6(p-1)}^{2\left(p-1\right)}}+\left\|\Delta\hat{u}+\left|\hat{u}\right|^{p-1}\hat{u}\right\|_{L^{2}(\Omega)}\right).\label{coroinequ1}
\end{align}
\end{coro}
\begin{proof}
Let us denote  $F(u)=\left|u\right|^{p-1}u$ in this proof.
Due to Theorem \ref{linf}, we have
\begin{align*}
\left\|u-\hat{u}\right\|_{L^{\infty}(\Omega)}&=\rho\left\|\omega\right\|_{L^{\infty}(\Omega)}\\
&\leq\rho\left(c_{0}\left\|\omega\right\|_{L^{2}(\Omega)}+c_{1}\left\|\omega\right\|_{V}+c_{2}\left\|\Delta\omega\right\|_{L^{2}(\Omega)}\right)\\
&\leq\rho\left(c_{0}C_{2}+c_{1}+c_{2}\left\|\Delta\omega\right\|_{L^{2}(\Omega)}\right).
\end{align*}
The last term $\left\|\Delta\omega\right\|_{L^{2}(\Omega)}$ is estimated by
\begin{align*}
\rho\left\|\Delta\omega\right\|_{L^{2}(\Omega)}&=\left\|F\left(\hat{u}+\rho\omega\right)+\Delta\hat{u}\right\|_{L^{2}(\Omega)}\\
&=\left\|F\left(\hat{u}+\rho\omega\right)-F\left(\hat{u}\right)+F\left(\hat{u}\right)+\Delta\hat{u}\right\|_{L^{2}(\Omega)}\\
&\leq\left\|F\left(\hat{u}+\rho\omega\right)-F\left(\hat{u}\right)\right\|_{L^{2}(\Omega)}+\left\|\Delta\hat{u}+F\left(\hat{u}\right)\right\|_{L^{2}(\Omega)}.
\end{align*}
Since the mean value theorem ensures that
\begin{align*}
&\displaystyle \int_{\Omega}\left(F\left(\hat{u}+\rho\omega\right)-F\left(\hat{u}\right)\right)^{2}dx\\
=&\displaystyle \int_{\Omega}\left(\rho\omega(x)\int_{0}^{1}F_{\hat{u}+t\rho\omega}'\left(x\right)dt\right)^{2}dx\\
=&\displaystyle \int_{\Omega}\left(\rho\omega(x)\int_{0}^{1}p\left|\hat{u}(x)+\rho t\omega(x)\right|^{p-1}dt\right)^{2}dx\\
=&p^{2}\displaystyle \rho^{2}\int_{\Omega}\omega(x)^{2}\left(\int_{0}^{1}\left|\hat{u}(x)+\rho t\omega(x)\right|^{p-1}dt\right)^{2}dx\\
\leq&p^{2}\displaystyle \rho^{2}\int_{\Omega}\omega(x)^{2}\int_{0}^{1}\left|\hat{u}(x)+\rho t\omega(x)\right|^{2(p-1)}dtdx\\
\leq&p^{2}\displaystyle \rho^{2}\left\|\omega\right\|_{L^{3}\left(\Omega\right)}^{2}\int_{0}^{1}\left\|\left|\hat{u}+\rho\omega t\right|^{2(p-1)}\right\|_{L^{3}\left(\Omega\right)}dt\\
=&p^{2}\displaystyle \rho^{2}\left\|\omega\right\|_{L^{3}\left(\Omega\right)}^{2}\int_{0}^{1}\left\|\hat{u}+\rho\omega t\right\|_{L^{6(p-1)}\left(\Omega\right)}^{2\left(p-1\right)}dt\\
\leq&p^{2}\displaystyle \rho^{2}\left\|\omega\right\|_{L^{3}\left(\Omega\right)}^{2}\int_{0}^{1}\left(\left\|\hat{u}\right\|_{L^{6(p-1)}\left(\Omega\right)}+t\rho\left\|\omega\right\|_{L^{6(p-1)}\left(\Omega\right)}\right)^{2\left(p-1\right)}dt\\
\leq&2^{2(p-1)-1}p^{2}\rho^{2}\left\|\omega\right\|_{L^{3}\left(\Omega\right)}^{2}\left\{\left\|\hat{u}\right\|_{L^{6(p-1)}\left(\Omega\right)}^{2\left(p-1\right)}+\int_{0}^{1}\left(t\rho\left\|\omega\right\|_{L^{6(p-1)}\left(\Omega\right)}\right)^{2\left(p-1\right)}dt\right\}\\
=&2^{2p-3}p^{2}\rho^{2}\left\|\omega\right\|_{L^{3}\left(\Omega\right)}^{2}\left(\left\|\hat{u}\right\|_{L^{6(p-1)}\left(\Omega\right)}^{2\left(p-1\right)}+\frac{\rho^{2(p-1)}}{2p-1}\left\|\omega\right\|_{L^{6(p-1)}\left(\Omega\right)}^{2\left(p-1\right)}\right)\\
\leq&2^{2p-3}p^{2}\rho^{2}C_{3}^{2}\left(\left\|\hat{u}\right\|_{L^{6(p-1)}\left(\Omega\right)}^{2\left(p-1\right)}+\frac{\rho^{2(p-1)}}{2p-1}C_{6(p-1)}^{2\left(p-1\right)}\right).
\end{align*}
it follows that
\begin{align*}
\rho\left\|\Delta\omega\right\|_{L^{2}(\Omega)}\leq 2^{p-\frac{3}{2}}p\rho C_{3}\sqrt{\left\|\hat{u}\right\|_{L^{6(p-1)}\left(\Omega\right)}^{2\left(p-1\right)}+\frac{\rho^{2(p-1)}}{2p-1}C_{6(p-1)}^{2\left(p-1\right)}}+\left\|\Delta\hat{u}+F\left(\hat{u}\right)\right\|_{L^{2}(\Omega)}.
\end{align*}
Consequently, the $L^{\infty}$ error of $u$ is estimated as asserted in \eqref{coroinequ1}.
\end{proof}

\begin{coro}\label{Linfcoro3}
Let $u$ be a solution to \eqref{allenabsproblem} satisfying \eqref{h10error} for $\hat{u}\in V$ such that $\Delta\hat{u}\in L^{2}\left(\Omega\right)$.
Moreover, let $c_{0},\ c_{1}$, and $c_{2}$ be as in Theorem {\rm \ref{linf}}.
Then,
\begin{align}
&\left\|u-\hat{u}\right\|_{L^{\infty}(\Omega)}\leq c_{0}C_{2}\rho+c_{1}\rho+\nonumber\\
&~~c_{2}\left(\rho\varepsilon^{-2}C_{3}\left(1+3\left\|\hat{u}\right\|_{L^{12}(\Omega)}^{2}+3\rho C_{12}\left\|\hat{u}\right\|_{L^{12}(\Omega)}+\rho^{2}C_{12}^{2}\right)+\left\|\Delta\hat{u}+\varepsilon^{-2}(\hat{u}-\hat{u}^{3})\right\|_{L^{2}(\Omega)}\right).\label{coroinequ3}
\end{align}
\end{coro}
\begin{proof}
Let us denote  $F(u)=\varepsilon^{-2}(u-u^{3})$ in this proof.
Due to Theorem \ref{linf}, we have
\begin{align*}
\left\|u-\hat{u}\right\|_{L^{\infty}(\Omega)}&=\rho\left\|\omega\right\|_{L^{\infty}(\Omega)}\\
&\leq\rho\left(c_{0}\left\|\omega\right\|_{L^{2}(\Omega)}+c_{1}\left\|\omega\right\|_{V}+c_{2}\left\|\Delta\omega\right\|_{L^{2}(\Omega)}\right)\\
&\leq\rho\left(c_{0}C_{2}+c_{1}+c_{2}\left\|\Delta\omega\right\|_{L^{2}(\Omega)}\right).
\end{align*}
The last term $\left\|\Delta\omega\right\|_{L^{2}(\Omega)}$ is estimated by
\begin{align*}
\rho\left\|\Delta\omega\right\|_{L^{2}(\Omega)}&=\left\|F\left(\hat{u}+\rho\omega\right)+\Delta\hat{u}\right\|_{L^{2}(\Omega)}\\
&=\left\|F\left(\hat{u}+\rho\omega\right)-F\left(\hat{u}\right)+F\left(\hat{u}\right)+\Delta\hat{u}\right\|_{L^{2}(\Omega)}\\
&\leq\left\|F\left(\hat{u}+\rho\omega\right)-F\left(\hat{u}\right)\right\|_{L^{2}(\Omega)}+\left\|\Delta\hat{u}+F\left(\hat{u}\right)\right\|_{L^{2}(\Omega)}.
\end{align*}
Since the mean value theorem ensures that
\begin{align*}
&\displaystyle \int_{\Omega}\left(F\left(\hat{u}+\rho\omega\right)-F\left(\hat{u}\right)\right)^{2}dx\\
=&\displaystyle \int_{\Omega}\left(\rho\omega(x)\int_{0}^{1}F'\left(\hat{u}(x)+t\rho\omega(x)\right)dt\right)^{2}dx\\
=&\displaystyle \int_{\Omega}\left(\rho\omega(x)\int_{0}^{1}\varepsilon^{-2}\left\{\left(1-3(\hat{u}(x)+t\rho\omega(x)\right)^{2}\right\}dt\right)^{2}dx\\
=&\displaystyle \rho^{2}\varepsilon^{-4}\int_{\Omega}\omega(x)^{2}\left(\int_{0}^{1}\left\{1-3(\hat{u}(x)+t\rho\omega(x))^{2}\right\}dt\right)^{2}dx\\
=&\displaystyle \rho^{2}\varepsilon^{-4}\int_{\Omega}\omega(x)^{2}\left(\int_{0}^{1}\left(1-3\hat{u}(x)^{2}-6t\rho\omega(x)\hat{u}(x)-3t^{2}\rho^{2}\omega(x)^{2}\right)dt\right)^{2}dx\\
\leq&\rho^{2}\varepsilon^{-4}\left\|\omega\right\|_{L^{3}(\Omega)}^{2}\left\|\left(\int_{0}^{1}\left(1-3\hat{u}^{2}-6t\rho\omega\hat{u}-3t^{2}\rho^{2}\omega^{2}\right)dt\right)^{2}\right\|_{L^{3}(\Omega)}\\
=&\rho^{2}\varepsilon^{-4}\left\|\omega\right\|_{L^{3}(\Omega)}^{2}\left\|\int_{0}^{1}\left(1-3\hat{u}^{2}-6t\rho\omega\hat{u}-3t^{2}\rho^{2}\omega^{2}\right)dt\right\|_{L^{6}(\Omega)}^{2}\\
\leq&\rho^{2}\varepsilon^{-4}\left\|\omega\right\|_{L^{3}(\Omega)}^{2}\left(1+3\left\|\hat{u}^{2}\right\|_{L^{6}(\Omega)}+3\rho\left\|\omega\hat{u}\right\|_{L^{6}(\Omega)}+\rho^{2}\left\|\omega^{2}\right\|_{L^{6}(\Omega)}\right)^{2}\\
\leq&\rho^{2}\varepsilon^{-4}\left\|\omega\right\|_{L^{3}(\Omega)}^{2}\left(1+3\left\|\hat{u}\right\|_{L^{12}(\Omega)}^{2}+3\rho\left\|\hat{u}\right\|_{L^{12}(\Omega)}\left\|\omega\right\|_{L^{12}(\Omega)}+\rho^{2}\left\|\omega\right\|_{L^{12}(\Omega)}^{2}\right)^{2}\\
\leq&\rho^{2}\varepsilon^{-4}C_{3}^{2}\left(1+3\left\|\hat{u}\right\|_{L^{12}(\Omega)}^{2}+3\rho C_{12}\left\|\hat{u}\right\|_{L^{12}(\Omega)}+\rho^{2}C_{12}^{2}\right)^{2}
\end{align*}
it follows that
\begin{align*}
\rho\left\|\Delta\omega\right\|_{L^{2}(\Omega)}\leq\rho\varepsilon^{-2}C_{3}\left(1+3\left\|\hat{u}\right\|_{L^{12}(\Omega)}^{2}+3\rho C_{12}\left\|\hat{u}\right\|_{L^{12}(\Omega)}+\rho^{2}C_{12}^{2}\right)+\left\|\Delta\hat{u}+F\left(\hat{u}\right)\right\|_{L^{2}(\Omega)}.
\end{align*}
Consequently, the $L^{\infty}$ error of $u$ is estimated as asserted in \eqref{coroinequ3}.
\end{proof}

%% file: embedding.tex
\def\thesection{Appendix~\Alph{section}}
\section{Simple bounds for the needed embedding constants}\label{sec:embedding}
\def\thesection{\Alph{section}}
The following theorem provides the best constant in the classical Sobolev inequality with critical exponents.
\begin{theo}[T.~Aubin \cite{aubin1976} and G.~Talenti \cite{talenti1976}]\label{talentitheo}
Let $u$ be any function in $W^{1,q}\left(\mathbb{R}^{n}\right)\ (n\geq 2)$, where $q$ is any real number such that $1<q<n$.
Moreover, set $p=nq/\left(n-q\right)$.
Then, $u \in L^{p}\left(\mathbb{R}^{n}\right)$ and
\begin{align*}
\left(\int_{\mathbb{R}^{n}}\left|u(x)\right|^{p}dx\right)^{\frac{1}{p}}\leq T_{p}\left(\int_{\mathbb{R}^{n}}\left|\nabla u(x)\right|_{2}^{q}dx\right)^{\frac{1}{q}}
\end{align*}
holds for
\begin{align}
T_{p}=\pi^{-\frac{1}{2}}n^{-\frac{1}{q}}\left(\frac{q-1}{n-q}\right)^{1-\frac{1}{q}}\left\{\frac{\Gamma\left(1+\frac{n}{2}\right)\Gamma\left(n\right)}{\Gamma\left(\frac{n}{q}\right)\Gamma\left(1+n-\frac{n}{q}\right)}\right\}^{\frac{1}{n}}\label{talenticonst},
\end{align}
where
$\left|\nabla u\right|_{2}=\left((\partial u/\partial x_{1})^{2}+(\partial u/\partial x_{2})^{2}+\cdots+(\partial u/\partial x_{n})^{2}\right)^{1/2}$,
and
$\Gamma$ denotes the gamma function.
\end{theo}

The following corollary, obtained from Theorem \ref{talentitheo}, provides a simple bound for the embedding constant from $H_{0}^{1}\left(\Omega\right)$ to $L^{p}(\Omega)$ for a bounded domain $\Omega$.

\begin{coro}\label{roughboundtheo}
Let $\Omega\subset \mathbb{R}^{n}\,(n\geq 2)$ be a bounded domain.
Let $p$ be a real number such that $p\in(n/(n-1),2n/(n-2)]$ if $n\geq 3$ and $p\in(n/(n-1),\infty)$ if $n=2$.
Moreover, set $q=np/(n+p).$
Then, $(\ref{embedding})$ holds for
\begin{align*}
	C_{p}\left(\Omega\right)=\left|\Omega\right|^{\frac{2-q}{2q}}T_{p},
\end{align*}
where $T_{p}$ is the constant in {\rm (\ref{talenticonst})}.
\end{coro}
\begin{proof}
By zero extension outside $\Omega$, we may regard $u\in H_{0}^{1}\left(\Omega\right)$ as an element $u\in W^{1,q}\left(\mathbb{R}^{n}\right)$; note that $q<2$.
Therefore, from Theorem \ref{talentitheo},
\begin{align}
	\left\|u\right\|_{L^{p}\left(\Omega\right)}
	\leq T_{p}\left(\int_{\Omega}\left|\nabla u\left(x\right)\right|_{2}^{q}dx\right)^{\frac{1}{q}}.\label{embedding/theo/1}
\end{align}
H\"{o}lder's inequality gives
\begin{align}
\int_{\Omega}\left|\nabla u\left(x\right)\right|_{2}^{q}dx&
\leq\left(\int_{\Omega}
\left|
	\nabla u\left(x\right)
\right|_{2}^{q\cdot\frac{2}{q}}dx\right)^{\frac{q}{2}}\left(\int_{\Omega}1^{\frac{2}{2-q}}dx\right)^{\frac{2-q}{2}}\nonumber\\
&=\left|\Omega\right|^{\frac{2-q}{2}}
	\left(\int_{\Omega}\left|
	\nabla u\left(x\right)
\right|_{2}^{2}dx\right)^{\frac{q}{2}},\nonumber
\end{align}
that is,
\begin{align}
    \left(\int_{\mathbb{R}^{n}}\left|\nabla u\left(x\right)\right|_{2}^{q}dx\right)^{\frac{1}{q}}
\leq\left|\Omega\right|^{\frac{2-q}{2q}}\left\|\nabla u\right\|_{L^{2}\left(\Omega\right)},\label{embedding/theo/2}
\end{align}
where $\left|\Omega\right|$ is the measure of $\Omega$.
From (\ref{embedding/theo/1}) and (\ref{embedding/theo/2}),
it follows that
\begin{align*}
    \left\|u\right\|_{L^{p}\left(\Omega\right)}&
\leq\left|\Omega\right|^{\frac{2-q}{2q}}T_{p}\left\|\nabla u\right\|_{L^{2}\left(\Omega\right)}.
\end{align*}
For any $\tau \geq 0$, it is true that $\left\|\nabla u\right\|_{L^{2}\left(\Omega\right)} \leq \left\|u\right\|_{V}$.
\end{proof}

\begin{rem}
The case that $p=2$ is ruled out in Corollary {\rm \ref{roughboundtheo}},
but it is well known that
\begin{align*}
\displaystyle \left\|u\right\|_{L^{2}(\Omega)}\leq\frac{1}{\sqrt{\lambda_{1}+\tau}}\left\|u\right\|_{V},
\end{align*}
where $\lambda_{1}$ is the first eigenvalue of the following problem:
\begin{align}
\left(\nabla u,\nabla v\right)=\lambda\left(u,v\right)~~{\rm for~all~}v\in V.\label{weaklapeig}
\end{align}
Note that, when $\Omega = (0,1)^2$, $\lambda_{1}=2\pi^{2}$.
\end{rem}

The use of the following theorem enables us to obtain an upper bound of the embedding constant when the first eigenvalue $\lambda_{1}$ of \eqref{weaklapeig} is concretely estimated.
We employ the smaller of the two estimations of $C_p$ derived from Corollary \ref{roughboundtheo} and Theorem \ref{plumembedding}.
\begin{theo}[\cite{plum2008}]\label{plumembedding}
Let $\lambda_{1}$ denote the first eigenvalue of the problem \eqref{weaklapeig}.\\
$a)$~~Let $n=2$ and $p\in[2,\infty).$
With the largest integer $\nu$ satisfying $\nu\leq p/2,\ (\ref{embedding})$ holds for
\begin{align*}
C_{p}\left(\Omega\right)=\left(\frac{1}{2}\right)^{\frac{1}{2}+\frac{2\nu-3}{p}}\left[\frac{p}{2}\left(\frac{p}{2}-1\right)\cdots\left(\frac{p}{2}-\nu+2\right)\right]^{\frac{2}{p}}\left(\lambda_{1}+\frac{p}{2}\tau\right)^{-\frac{1}{p}},
\end{align*}
where $\displaystyle \frac{p}{2}\left(\frac{p}{2}-1\right)\cdots\left(\frac{p}{2}-\nu+2\right)=1$ if $\nu=1.$\\[1pt]
$b)$~~Let $n\geq 3$ and $p\in[2,2n/(n-2)]$.
With $s:=n(p^{-1}-2^{-1}+n^{-1})\in[0,1],\ (\ref{embedding})$ holds for
\begin{align*}
C_{p}\left(\Omega\right)=\left(\frac{n-1}{\sqrt{n}\left(n-2\right)}\right)^{1-s}\left(\frac{s}{s\lambda_{1}+\tau}\right)^{\frac{s}{2}}.
\end{align*}
\end{theo}